%% file: Revision-SIAM.tex
\documentclass[onefignum,onetabnum, a4paper]{siamart190516}

\usepackage{amsmath}
\usepackage{amsfonts}
\usepackage{amssymb}

\newcommand{\R}{{\cal R}}

\newcommand{\cA}{{\mathfrak A}}
\newcommand{\cS}{{\mathfrak S}}
\newcommand{\cX}{{\mathfrak X}}

\newcommand{\cP}{{\mathfrak P}}

\newcommand{\cM}{{\mathfrak M}}

\newcommand{\cF}{{\mathfrak F}}

\newcommand{\cR}{{\mathfrak R}}
\newcommand{\cV}{{\mathfrak V}}

\def\argmin{\mathop{\rm argmin}}

\newcommand{\Q}{{\cal Q}}

\newcommand{\V}{{\cal V}}
\newcommand{\Z}{{\cal Z}}
\newcommand{\F}{{\cal F}}

\newcommand{\A}{{\cal A}}
\newcommand{\B}{{\cal B}}
\newcommand{\G}{{\cal G}}

\newcommand{\J}{{\cal J}}
\newcommand{\M}{{\cal M}}
\newcommand{\X}{{\cal X}}

\newcommand{\LL}{{L}}

\newcommand{\be}{\begin{equation}}
\newcommand{\ee}{\end{equation}}

\newcommand{\lan}{\langle}
\newcommand{\ran}{\rangle}

\def\w{\omega}
\def\e{\varepsilon}
\def\O{\Omega}

\def\supp {{\rm supp}}

\makeatletter
\newif\ifrev@inenv		
\ifdefined\AddToHook
	\AddToHook{env/rev/begin}{\rev@inenvtrue}
\else
	\usepackage{etoolbox}
	\AtBeginEnvironment{rev}{\rev@inenvtrue}
\fi
\newcommand\rev[1][black]{
	\ifrev@inenv
		\rev@inenvfalse
		\expandafter\color
	\else\expandafter\textcolor
	\fi
	{#1}%
}
\makeatother

\DeclareMathOperator*{\ess}{ess\,sup}
\DeclareMathOperator*{\essinf}{ess\,inf}
\DeclareMathOperator*{\esssup}{ess\,sup}

\DeclareMathOperator{\avr}{{\sf AV@R}}
\DeclareMathOperator{\AVaR}{{\sf AV@R}}

\def\nd {{\sf nd}}

\def\bbr{{\mathbb R}} 
\DeclareMathOperator{\E}{\mathbb E}		
\def\bbe{\E} 
\def\bbn{{\mathbb{N}}}

\def\bB{{\mathbb{B}}}
\def\eR{\overline{\mathbb{R}}}

\def\ebr{\overline{\mathbb{R}}}
\newcommand{\ind}{{\mbox{\boldmath$1$}}}

\newtheorem{example}	[theorem]{Example}

\input{ex_shared}

\ifpdf
\hypersetup{
  pdftitle={Mathematical Foundations of Distributionally Robust  Multistage Optimization},
  pdfauthor={Alois Pichler and Alexander Shapiro}
}
\fi

\begin{document}

\maketitle

\begin{abstract}
	Distributionally robust optimization involves various probability measures in its problem formulation. They can be bundled to constitute a risk functional. For this equivalence, risk functionals constitute a fundamental building block in distributionally robust stochastic programming.
	Multistage programming requires conditional versions of risk functionals to re-assess future risk after partial realizations and after preceding decisions.

	This paper discusses a  construction of the conditional counterpart of a risk functional  by passing its genuine characteristics to its conditional counterparts.  The conditional risk functionals  turn out to be  different  from the nested analogues  of the  original (law invariant) risk measure. It is demonstrated that the initial measure and its nested decomposition can be used in a distributionally robust  multistage  setting.
\end{abstract}



\section{Introduction}
	This paper addresses  \emph{Distributionally Robust Optimization} (DRO) when applied to multistage stochastic programming problems. This robust approach involves various probability measures and discloses an optimal solution which is robust with respect to varying distributions.
	DRO involves the functional
	\begin{equation}\label{fun-1}
	 	\R(Z)  :=   \sup_{Q\in \cM}\bbe_Q[Z]
	\end{equation}
	for a chosen so-called \emph{ambiguity} set $\cM$ of probability measures.
	Various constructions of the ambiguity  set  were suggested  in the applications. \rev{It is beyond the scope of this paper to give a survey of this literature, we can mention for example \cite{Kuhn2015, Wiesemann2014, KuhnWiesemannFoundations}  and references therein.} By duality, DRO   is closely related to the \emph{risk averse} approach where~$\R$ is viewed as a risk measure \rev{(cf.\ \cite{ADEH:1999, Follmer2004, RuszczynskiShapiro2009}).}
	In case the ambiguity set $\cM=\{P\}$ is a singleton this reduces to the risk neutral approach. The essential difference between the DRO and the risk averse approaches  is how the corresponding functional~$\R$ is defined.

	In the static setting the equation~\eqref{fun-1} of the objective   functional~$\R$  is reasonably straightforward.  Multistage problems with  decisions made dynamically, on the other hand,  are considerably more involved.
	\begin{rev}In the risk averse approach the corresponding risk measure, associated with the sequential decision process, usually  is formulated in the nested (composite) form (cf.\ \cite{FriRos:2005, Riedel2004, Ruszczynski}). Such nested formulations are   amendable to Bellman's principle of optimality and suitable for  writing  the respective dynamic programming equations. This  in turn is directly related to the concept of time consistency meaning  that a decision maker should not reconsider optimality of the   decisions at the later  stages of the process. In that respect nested formulations provide a natural framework for  the optimality criteria at every stage of the   process.

In the DRO  setting the situation is  more delicate.
 It turns out that seemingly natural extensions of the DRO and the risk averse approaches could  lead to somewhat different frameworks. The basic question of formulating  a conditional counterpart of the DRO functional~\eqref{fun-1} turns out to be not straightforward. The main goal  of this paper is to elaborate on the involved  fundamental  questions    in a mathematical context.
\end{rev}

\paragraph{Outline}
 The following section elaborates the mathematical setting for distributionally robust functionals and  risk measures. Section~\ref{sec-condcounter} addresses their  conditional counterparts. Nested formulations are important in multistage stochastic optimization, they are considered in Section~\ref{sec-nested}. The corresponding  dynamic programming  equations are elaborated in Section~\ref{sec:MSO} before we   conclude in Section~\ref{sec:Summary}.

\section{Preliminaries and the static setting}\label{sec:basic}

Let $(\O,\F)$ be a measurable space and $\cM$ be a nonempty set of probability measures (distributions) on $(\O,\F)$. We denote by $\cP$ the set of all probability distributions on $(\O,\F)$, thus  $\cM\subset \cP$.
The corresponding functional
\begin{equation}\label{fun-2}
	\R(Z)  :=   \sup_{Q\in \cM}\bbe_Q[Z], \quad Z\in \Z,
\end{equation}
is defined on a linear space $\Z$ of measurable functions   $Z\colon\O\to \bbr$, where we use the notation $\bbe_Q[Z]$ to emphasize that the expectation
is with respect to the probability measure~$Q\in \cP$.
 We refer to $\cM$ as the \emph{ambiguity set} of probability measures and to $\R$ as the \emph{distributionally robust} functional.

We assume that $\Z$ is a linear (vector) space of measurable functions such that $\ind_\O\in \Z$, where $\ind_A$ is the indicator  function of the set $A$, i.e., $\ind_A (x)=1$ for $x\in A$ and $\ind_A (x)=0$ otherwise.
This implies that $Z+a\in \Z$, if $Z\in \Z$ and $a\in \bbr$.
We associate with $\Z$ a linear space $\Z^*$ of finite signed measures
 on $(\O, \F)$ such that $\Z$ and $\Z^*$ are paired
 vector spaces with respect to the bilinear form
\begin{equation}\label{bilform-1}
	\lan Z,Q\ran:=   \int_\O Z(\w)\,Q(d\w),\quad Z\in \Z,\  Q\in \Z^*.
\end{equation}
\rev{The integral in~\eqref{bilform-1} is supposed to be well-defined and finite valued for all $Z\in \Z$ and $Q\in \Z^*$. }
We assume that $\cM\subset \Z^*$ and hence, for any $Z\in \Z$ and $Q\in \cM$, the expectation   $\bbe_Q[Z]= \lan Z,Q\ran$ is well-defined and finite valued.

Stochastic optimization   problems are often formulated by involving continuous probability measures, but computations need to be implemented with discrete measures. It is thus essential to elaborate on both, discrete and continuous probability measures. For this reason we consider and address the following settings separately before combining them in the general framework.

\subsection{Random variables with finite moments}
\label{ex-lp}

	Consider the
 space $\Z:=L_p(\O,\F,P)$, $p\in [1,\infty)$,  of random variables $Z:\O\to\bbr$ with finite $p$-th order moments.
 We refer to~$P$ as the {\em reference} probability measure.
The space  $\Z$, equipped with the corresponding norm $\|\cdot\|_p$,  is a Banach space. Its dual space (of continuous linear functionals) is $L_q(\O,\F,P)$  such that $1/p+1/q=1$, $q\in (1,\infty]$. Suppose  that  the set $\cM$ consists of probability measures $Q$ absolutely continuous with respect to~$P$  such that their density $\zeta=dQ/dP$ belongs to the space $L_q(\O,\F,P)$.

	Specifically, we pair  $\Z=L_p(\O,\F,P)$ with its dual $\Z^*= L_q(\O,\F,P)$  with the respective bilinear form
	\begin{equation}\label{bilform-2}
		\lan Z,\zeta\ran:=   \int_\O Z(\w)\zeta (\w)\,P(d\w),\quad Z\in \Z,\ \zeta\in \Z^*.
	\end{equation}
	Note that here $Z\in \Z$ is actually a class of measurable variables which can be different from each other on sets of $P$-measure zero. We have that the respective  functional $\R(\cdot)$ is finite valued on $\Z$ iff the set
\begin{equation}\label{dual-r1}
  \cA:= \{\zeta=dQ/dP\colon Q\in \cM\}
\end{equation}
 is bounded in the corresponding norm $\|\cdot\|_q$. The set $\cA$ can be assumed to be convex and closed in the weak$^*$ topology of the dual space $\Z^*$. We refer to $\cA$ as the dual set of the risk functional $\R$.

\subsection{Continuous random variables}\label{ex-comeg}
Suppose that 	the set $\O$ is a compact metric  space and $\F$ is its  Borel $\sigma$\nobreakdash-algebra. Let $\Z:= C(\O)$ be the space of continuous functions $Z\colon\O\to\bbr$ equipped with the supremum norm $\|Z\|_\infty=\sup_{\w\in \O}|Z(\w)|$. The dual space of $C(\O)$ is formed by finite signed measures on $(\O,\F)$ with respect to the bilinear form~\eqref{bilform-1}  (Riesz representation).
	It could be noted that the dual norm $\|\cdot\|^*$  of $\Z^*=C(\O)^*$ is the total variation norm and $\|Q\|^*=1$ for any probability measure $Q\in \cP$. Also, for any $Q\in \cP$ and $Z$, $Z'\in \Z$ we have that
	\[	\bbe_Q[Z']=\bbe_Q[Z+Z'-Z]\le \bbe_Q[Z]+a,	\]
	where $a:= \sup_{\w\in \O} |Z^\prime(\w)-Z(\w)|=\|Z^\prime-Z\|_\infty$.
	It follows that $\R(\cdot)$, as defined in~\eqref{fun-2}, is finite valued and
	\begin{equation}\label{lipsch}
		|\R(Z')-\R(Z)|\le \|Z'-Z\|_\infty,
	\end{equation}
	i.e., $\R$ is Lipschitz continuous with Lipschitz constant~$1$.

\subsection{Bounded functions}\label{ex-comeg1}
Consider the   space $\Z:= \bB(\O)$ consisting  of measurable and bounded functions $Z\colon\O\to\bbr$. Equipped with the supremum norm $\|\cdot\|_\infty$ this becomes a Banach space. Its dual (of continuous linear functionals) is quite complicated.
Nevertheless, we can pair it with the space $\Z^*$ of finite signed measures on $(\O,\F)$ employing the bilinear form~\eqref{bilform-1}.
As in the previous example, it can be shown that the inequality~\eqref{lipsch} holds for any $Z$, $Z'\in \bB(\O)$, and hence the corresponding distributionally robust functional~\eqref{fun-2} is finite valued and continuous in the norm $\|\cdot\|_\infty$ topology. However, we will have to deal here with the paired topologies of the spaces $\Z$ and $\Z^*$.

\medskip
There is a natural order relation making the space $\Z$ from any of the three subsections above a partially ordered set. For $Z$,~$Z'\in \Z$ we write $Z'\succeq Z$ (or $Z\preceq Z'$) if $Z'\ge Z$ almost surely with respect to the reference measure~$P$ in the $L_p$ setting, and $Z'(\w)\ge Z(\w)$ for all $\w\in \O$ in the $C(\O)$ and $\bB(\O)$ settings.

\subsection{Correspondence of robust functionals with risk functionals}

It is straightforward to verify that the distributionally robust functional~\eqref{fun-2} satisfies the following conditions for any $Z$, $Z'\in \Z$:
\begin{enumerate}
	\item \label{enu:A1} \emph{subadditivity:}
	$\R(Z + Z') \le \R(Z) +\R(Z')$,
	\item \label{enu:A2} \emph{monotonicity:} if $Z' \succeq Z$,
	then $\R(Z')\geq \R(Z)$,
	\item \label{enu:A3} \emph{translation equivariance:} if $a\in \bbr$, then $\R(Z+a)=\R(Z)+a$,
	\item \label{enu:A4} \emph{positive homogeneity:} if $\lambda >0$, then
	$\R(\lambda\, Z)=\lambda\,\R(Z)$.
\end{enumerate}

Conversely, in the $L_p$ and $C(\O)$ frameworks,
any \emph{real valued} functional $\R\colon\Z\to \bbr$ satisfying the axioms~\ref{enu:A1}--\ref{enu:A4} is continuous and can be represented in the dual form~\eqref{fun-1}
for an appropriate set $\cM$ of probability measures  (cf.\ \cite[Proposition~3.1]{Ruszczynski2006}).
For the $\bB(\O)$ setting this is more involved, in order to guarantee the dual representation it should be also verified that~$\R$ is lower semi-continuous with respect to the considered paired topologies.

\medskip
In the risk averse perspective the most important example of the functional~$\R$ is the Average Value-at-Risk
\begin{equation}\label{avr-1}
	\avr_\alpha (Z)=
	\inf_{\tau\in \bbr}\left\{\tau+(1-\alpha)^{-1}\bbe_P[Z-\tau]_+\right\}, \;\alpha\in [0,1),
\end{equation}
with  $\Z=L_1(\O,\F,P)$. \rev{In some publications it is also called Conditional Value-at-Risk, Expected Shortfall, Expected Tail Loss; variational representation~\eqref{avr-1} is due to Rockafellar and  Uryasev
\cite{RockafellarUryasev2000}.}
In the dual form, it   has the representation~\eqref{fun-2} with the corresponding dual set
\begin{equation}\label{avr-2}
	\cA=\left \{\zeta\colon 0\le \zeta\le 1/(1-\alpha),\text{ and } \int_\O \zeta\,dP=1\right\},
\end{equation}
which constitute densities with respect to the reference measure $P$.
For  $\alpha=1$, $\AVaR_1(Z)=\esssup(Z)$ is the essential supremum.

\section{Conditional counterparts of distributionally robust functionals}
\label{sec-condcounter}

	To set up the distributionally robust optimization in a multistage framework we need to define a conditional counterpart of the robust functional~$\R$
	defined in~\eqref{fun-2}. To this end let $\G$ be a   $\sigma$\nobreakdash-subalgebra of~$\F$. Then for any $Q\in \cM$ we can consider the corresponding conditional expectation $\bbe_{Q|\G}[Z]$ of a random variable $Z\in \Z$
	(we can refer  to  \cite{Kallenberg2002Foundations}, e.g.,  for a rigorous discussion of the  precise meaning of conditional expectations).

	\subsection{Conditional discrepancy of risk functionals}\label{ex-avr} 
	A general caveat arises with a rigorous  definition of   conditional functionals, i.e., with functionals conditioned on    $\sigma$-subalgebras.
	To elaborate the disparity we consider a conditional counterpart of the risk functional~\eqref{fun-2} first and a variant next, which is most common in the literature. Both approaches seem natural in \rev{convenient} contexts, but they differ essentially.

	\subsubsection{The conditional analogue}
	Consider  the conditional counterpart
	\begin{equation}\label{eq:32}
		\R_{|\G}(Z)=\esssup_{Q\in \cM} \E_{Q|\G}[Z]
	\end{equation}
	of~\eqref{fun-2}. A rigorous meaning of~\eqref{eq:32} is given in Section~\ref{sec-condrisk} below; however, the inconsistency already occurs in a finite setting in which the essential supremum in~\eqref{eq:32} is a usual maximum.

\begin{definition}\label{def:partition}
	We say that a family $\{A_i\}_{i\in I}$ of nonempty  sets $A_i\in \F$ is a \emph{partition} of a set $A\in \F$  if  $\cup_{i\in I} A_i=A$ and $A_i\cap A_j=\emptyset$ for $i\ne j$. It is said  that the partition is countable (finite) if the index set $I$ is countable (finite).
\end{definition}

\begin{example}
\label{ex-finite}
{\rm
	Suppose the set $\O=\{\w_1,\dots,\w_n\}$ is finite, equipped with the $\sigma$\nobreakdash-algebra~$\F$ of all its subsets and the reference probability  measure~$P$ assigns
	probabilities $p_i>0$ to each $\w_i\in \O$, $i\in \{1,\dots,n\}$. Then any probability measure $Q$ on $(\O,\F)$  is absolutely continuous with respect to~$P$,  and with a subalgebra $\G$ of $\F$  is associated  a finite partition $\{\Upsilon_i\}_{i\in I}$ of~$\O$ such that a variable $Z\colon\O\to \bbr$ is $\G$-measurable iff  it is constant on every $\Upsilon_i$,  $i\in I$.

	Now suppose that the ambiguity set $\cM$ has the following property:
\begin{equation}
	\begin{array}{lll}
		&&  \text{For every $i\in I$ and every  $\bar{\w}\in\Upsilon_i$ there is a measure $Q\in\cM$ such that:} \\
		&&\text{  $Q(\{\bar{\w}\})>0$ and
					$Q(\{\w\})=0$ for all $\w \in\Upsilon_i\setminus\{\bar{\w}\}$.}
		\label{property}
	\end{array}
\end{equation}
For such a measure~$Q$   the conditional expectation~$\bbe_{Q|\G}[Z](\cdot)$ is constant  on the set~$\Upsilon_i$ with the corresponding value
\begin{equation}
\E_Q[Z\, \ind_{\Upsilon_i}]/ Q(\Upsilon_i)=Z(\bar{\w}).
\end{equation}
	It follows with~\eqref{eq:32} that  the corresponding conditional risk functional is
	\begin{equation}\label{cond-av}
		\R_{|\G}(Z)(\w)=\max_{\w'\in \Upsilon_i} Z(\w'), \quad \w\in \Upsilon_i.
	\end{equation}
	That is, $ \R_{|\G}(Z)$ is the conditional supremum of $Z$, conditioned on the atoms~$\Upsilon_i$ of the $\sigma$\nobreakdash-algebra~$\G$. Note that  the conditional risk functional~\eqref{cond-av} is  independent of the reference measure~$P$, provided that $P(\{\w\})>0$ for all $\w\in\O$.
	
	The property~\eqref{property}  above is not uncommon.
Consider, for example, the   functional $\R:=\avr_\alpha$, $\alpha\in (0,1)$. Its dual representation~\eqref{avr-2}, of the  ambiguity (dual) set~$\cA$,
has the property~\eqref{property}  if $P(\Upsilon_i)\le \alpha$ for all $i\in I$,  i.e., if the atoms of~$\G$ are small relative to~$\alpha$.
\rev{Indeed, the   set $\cA$ consists of densities~$\zeta$ such that $0\le \zeta(\w)\le (1-\alpha)^{-1}$ for all $\w\in \O$,  and $\sum_{\w\in \O}\zeta(\w) P(\{\w\})=1$. For $i\in I$ consider  $\Upsilon_i$ and a point  $\bar{\w}\in \Upsilon_i$.   Choose $\zeta\in \cA$ such that
$\zeta(\bar{\w})=1$, and  $\zeta(\w)=0$ for $\w\in \Upsilon_i\setminus\{\bar{\w})$  and $\zeta(\w):= \kappa$
for $\w\in\O\setminus\Upsilon_i$ with $\kappa:=\frac{1-{P(\{\bar{\w}}\})}{1-P(\Upsilon_i)}$. Note that $0\le \kappa\le (1-\alpha)^{-1}$ since $P(\Upsilon_i)\le \alpha$.
 The probability measure corresponding to this density is  the required measure $Q$.
 }
 }
\end{example}

\subsubsection{The law invariant analogue}\label{disc}
The   definition~\eqref{avr-1} of the   $\avr_\alpha$   leads to the following  definition of its conditional counterpart  (e.g., \cite{Ruszczynski})
\begin{equation}\label{avr-cond1}
		\avr_{\alpha|\G} (Z)=
		\essinf_{Y\in L_1(\O,\G,P)}\left\{Y+(1-\alpha)^{-1}\bbe_{P|\G}[Z-Y]_+\right\}.
	\end{equation}
\rev{The above Example~\ref{ex-finite} demonstrates the difference between the nested $\avr$, defined in~\eqref{avr-cond1},  and the distributionally robust counterpart of $\avr$ obtained from the dual representation~\eqref{avr-2}.}
	In contrast to formula~\eqref{cond-av}, the nested  $\avr_{\alpha|\G} (Z)$  takes the value given by the respective Average Value-at-Risk on every set~$\Upsilon_i$. \rev{The  approach~\eqref{eq:32}  seems to be natural from the distributional robustness point of view while it is different from the nested approach to (law invariant) risk measures. We are going to discuss this in the next section.}

\subsection{The ``right"  conditional analogue}
	In the risk averse approach it is natural to consider law  invariant risk measures such that the value $\R(Z)$ depends only  on the distribution   (with respect to the reference probability measure) of $Z$. Then intuitively the  nested conditional counterpart   is defined in the same way as a function of the conditional distribution of~$P$. The definition~\eqref{avr-cond1} of the  nested Average Value-at-Risk  is of that form.
	
	On the other hand, suppose for example  that the ambiguity  set $\cM$ is given by convex combinations of a finite   set $\{Q_1,\dots,Q_m\}$ of probability measures. Then the definition via~\eqref{eq:32} makes sense. Also in some cases the functional $\R$ is not law invariant, and consequently its law invariant conditional counterpart cannot be defined.
	So the question of what is a \emph{right} definition of the conditional counterpart of the distributionally robust functional is open for discussion.
The nested approach  was discussed extensively in the   literature on risk averse  stochastic optimization \rev{(e.g., \cite{Riedel2004, Ruszczynski}).} On the other hand  the conditional counterpart~\eqref{eq:32} appears naturally as an extension of the definition~\eqref{fun-2} of the distributionally robust functional \rev{(e.g., \cite{Shapiro2015}).}   In Section~\ref{sec-rec} we will discuss the \emph{rectangular} setting where both approaches are equivalent.
Unless stated otherwise we deal in the remainder 
	with the conditional distributed robust  functional $\R_{|\G}$  stated in~\eqref{eq:32} above.

	To address the mathematical foundation we need to give a rigorous meaning to the setting~\eqref{eq:32} for uncountable measurable spaces.
It is tempting to define the conditional counterpart of $\R(Z)$ as the pointwise supremum,
	\begin{equation}\label{muldisr-1}
		\R_{|\G}(Z)(\w):= \sup_{Q\in \cM} \bbe_{Q|\G}[Z](\w),\quad\w\in \O.
	\end{equation}
	However, there are several technical problems with a rigorous meaning of the right-hand side of~\eqref{muldisr-1}.  Actually  $\bbe_{Q|\G}[Z]$  is a class of $\G$-measurable functions such that its versions can be different from each other on a set of $Q$-measure zero. So it is not clear what the supremum of such functions over $Q\in\cM$ really means; even if properly defined, its measurability is not obvious when the set $\cM$ is uncountable.
We will need the  concepts of essential supremum and infimum, discussed in the next section.


\subsection{The essential supremum}\label{sec:essent}
Unless stated differently 
we assume in this section  that there is a reference probability measure~$P$ on $(\O,\F)$ and that every $Q\in \cM$  is absolutely continuous with respect to~$P$. We use the notation $Q\ll P$ to denote that~$Q$  is absolutely continuous with respect to~$P$. We employ  the concept of essential supremum.

Let us start with  some basic definitions of partially ordered sets. A binary relation~$\preceq$ on a set~$\cS$ is said to be a \emph{partial order} on $\cS$
if for any $x$, $x'$, $x''\in \cS$ the following properties hold:
(i) $x\preceq x$,
(ii) if $x\preceq x'$ and $x'\preceq x$, then $x=x'$,
(iii)  if $x\preceq x'$ and $x'\preceq x''$, then $x\preceq x''$.
It said that an element $y\in \cS$ is an \emph{upper bound} for a (nonempty) set $\cV\subset \cS$ if $x\preceq y$ for any $x\in \cV$. If an upper bound $y$ of ~$\cV$ belongs to $\cV$, then $y$ is said to be the \emph{largest element} of $\cV$. The lower bound and the smallest element of $\cV$ are defined in the analogous way.
Note that by the property (ii)  of the partial order, it follows that if the largest (smallest) element exists, then it is unique. The least upper bound of $\cV$, which is the smallest element of the set of the upper bounds of $\cV$, is
called the supremum of $\cV$ and denoted $\sup \cV$. That is, $\sup \cV$ is an element of $\cS$ such that $x\preceq \sup \cV$ for any $x\in \cV$, and if $y$ is an upper bound of $\cV$, then $\sup \cV \preceq y$. If the supremum $\sup \cV$ exists, then it is unique.

Clearly if the set $\cV$ has the largest element $\bar{x}$, then $\bar{x}=\sup\cV$.
With the set $\cV$ we associate its order complement
\begin{equation}\label{complem}
	\cV^c:= \{y\in \cS\colon x\preceq y \text{ for all } x\in \cV\}.
\end{equation}
It is straightforward to verify the following well known result.

\begin{lemma}\label{lem-sup}
	The supremum $\sup \cV$ exists iff its order complement set $\cV^c$ has the smallest element $\bar{y}$, in which case $\bar{y}=\sup\cV$.
\end{lemma}

For random variables $X$, $Y$ we can define the partial order  $X\preceq Y$
meaning that $X\le Y$  almost surely with respect to the reference measure $P$.
In the considered setting  the supremum of a set $\cX$ of measurable functions is called the essential supremum of $\cX$. Denote by $\eR:= \mathbb R\cup\{\pm\infty\}$ the extended real line.

\begin{definition}\label{def-essup}
	The \emph{essential supremum} of a (nonempty) set $\cX$ of measurable functions $X\colon\O\to \ebr$, denoted $\esssup \cX$, is a measurable function
	$X^*$ satisfying the following properties: {\rm (i)}
 for any $X\in\cX$ it follows that $X\preceq X^*$,
{\rm (ii)} if $Y$ is a measurable function such that $X\preceq Y$ for all $X\in \cX$, then $X^*\preceq Y$.
\end{definition}

The \emph{essential infimum} is defined in the analogous way and is denoted $\essinf \cX$, that is $\essinf \cX:=-\esssup(-\cX).$  Let us emphasize that these concepts of the essential  supremum and infimum are defined with respect to the reference measure~$P$.
In the setting of Definition~\ref{def-essup}  the essential supremum always exists and,
as it was pointed above,   is unique (e.g., \cite[Appendix~A.5]{Follmer2004}).


\begin{proposition}\label{pr-essexist}
Let $\cX=\{X_i\}_{i\in I}$ be a nonempty  set of measurable functions $X_i\colon\O\to \ebr$. Then
the essential supremum $X^*=\esssup\cX$  exists and, moreover, there is a countable set
$J\subset I$ such that $P$\nobreakdash-almost surely $X^*= \sup_{i\in J}X_i$.
\end{proposition}

\begin{remark}
\label{rem-refmes}
In some settings there is no  natural reference probability measure.
In such cases we can consider   the settings of $\Z=C(\O)$ or $\Z=\bB(\O)$. Let $\cX\subset \Z$ be a nonempty set of functions and $\cX^c$ be its order complement.
Consider the pointwise partial order $X\preceq Y$,
meaning that $X(\w)\le Y(\w)$ for all $\w\in \O$, and the
respective pointwise supremum and pointwise infimum functions
\[
	X^*(\w):= \sup_{X\in \cX} X(\w) \text{ and } \bar{X}(\w):= \inf_{X\in \cX^c} X(\w),\;\w\in \O.
\]
It is not difficult to see that the set $\cX$ has the largest element iff $X^*\in \Z$ (i.e., $X^*$ is continuous in the $C(\O)$ setting and measurable in the $\bB(\O)$ setting), in which case $X^*=\esssup \cX$. Also, by Lemma~\ref{lem-sup} we have that the $\esssup \cX$ exists iff $\bar{X}$ belongs to $\Z$, in which case $\bar{X}(\w)=\esssup \cX$.

For example,
	consider the $C(\O)$-setting with $\O:= [0,1]$. Let $[a,b]$ be an interval such that $0<a\le b<1$, and $\cX$ be the set of functions $X\colon[0,1]\to \bbr$ such that $0\le X(\w)\le 1$, $\w\in [0,1]$, $X(\w)=0$, $\w\in [a,b]$. Here the pointwise supremum function $X^*$ has the form $X^*(\w)=0$ for $\w\in [a,b]$, and $X^*(\w)=1$ otherwise. Since $X^*$ is not a continuous function, the set $\cX$ does not have the largest element.
	Consider now the order complement $\cX^c$ and the corresponding pointwise infimum function $\bar{X}$.
	If $a<b$, then $\bar{X}(\w)=0$ for $\w\in (a,b)$, and $\bar{X}(\w)=1$ otherwise. In that case the set $\cX^c$ does not have the smallest element and the essential supremum of the set $\cX$ does not exist. On the other hand if $a=b$, then the constant function $\bar{X}(\cdot)\equiv 1$ is the smallest element of $\cX^c$ and is the essential supremum of~$\cX$.
\end{remark}

\subsection{Conditioning risk functionals}
\label{sec-condrisk}
This section resumes the $L_p$-setting described in Section~\ref{ex-lp} above; that is, $\Z=L_p(\O,\F,P)$.
It could be noted that conditional expectations with respect to different probability measures could have arbitrary  values on sets having  positive reference probability measure.
The following remark elaborates the consequences for the conditional risk functional.

\begin{remark}\label{ex-sigmalg}
	Let $P$ be a reference probability measure on $(\O,\F)$.
	Suppose that $\G=\F$.
	For any $Q\in \cM$ and $\F$-measurable variable $Z\colon\O\to\bbr$, we have that $Q$-almost surely $\bbe_{Q|\F}[Z]=Z$.
	If $\R_{|\G}$ is viewed as a mapping from $\LL_p(\O,\F,P)$ to $\LL_p(\O,\G,P)$, then   the respective translation equivariance axiom (cf.~\ref{enu:A3}) requires  that $P$-almost surely $\R_{|\G}(Y)=Y$ for any $\G$-measurable  $Y$. When $\G=\F$ this means that $P$-almost surely  $\R_{|\G}(Z)=Z$ for any $Z\in \Z$.
	On the other hand, it could happen for some $Q\in \cM$ and $A\in\F$ that $Q(A)=0$ while $P(A)>0$, i.e., $P$ is not absolutely continuous with respect to~$Q$.
	In that case a version of $\bbe_{Q|\F}[Z]$ can have arbitrary values on the set~$A$,
  and hence is not equal~$Z$ almost surely with respect to the reference measure~$P$.
\end{remark}

 In order to provide a definition of conditional distributionally robust functio\-nals we use the concept of essential supremum applied to the
 family $\left\{\bbe_{Q|\G}[Z](\cdot)\colon Q\in \cM\right \}$ of $\G$\nobreakdash-measurable functions.
 For that  we need to choose  a specific  version of $\bbe_{Q|\G}[Z]$ for a given $Z\in \Z$.  If~$P$ is absolutely continuous with respect to \emph{every} $Q\in \cM$,  then we can use any version of $\bbe_{Q|\G}[Z]$.  However, there are natural  examples where this does not hold for some $Q\in \cM$. In order to deal with this we consider the essential infimum (taken with respect to~$P$)  of the set of   versions of  $\bbe_{Q|\G}[Z]$, denoted
 $\essinf  \bbe_{Q|\G}[Z]$.
\begin{definition}\label{def:essinf}
	For a probability  measure $Q$ and a random variable~$Z$ we define
\[\essinf \E_{Q|\mathcal G}[Z]\] (taken with respect to $P$) as the essential infimum of the set of
versions of  $\bbe_{Q|\G}[Z]$.
\end{definition}

The $\essinf  \bbe_{Q|\G}[Z]$ takes the value $-\infty$ on any  set $A\in \G$ such that $Q(A)=0$ while $P(A)>0$.
If~$P$ is absolutely continuous with respect to $Q$, then
$\essinf  \bbe_{Q|\G}[Z]= \bbe_{Q|\G}[Z]$.
The essential infimum $\essinf  \bbe_{Q|\G}[Z]$  exists  for every $Q\in \cM$. However, for $Z\in L_p(\O,\F,P)$, the corresponding  $\essinf  \bbe_{Q|\G}[Z]$ belongs to $L_p(\O,\F,P)$ only if it takes the value $-\infty$ on a set of $P$-measure zero, that is only if $P\ll Q$.

The essential supremum of the family $\left\{\essinf \bbe_{Q|\G}[Z], \;Q\in \cM\right \}$ leads to the following definition of the conditional functional  $\R_{|\G}$. As before, $X\preceq Y$
means that $X\le Y$  almost surely with respect to the reference measure~$P$.

\begin{definition}[Conditional distributionally robust functional]\label{def-sup-1}
Let $\G$ be a   $\sigma$-subalgebra of~$\F$. A version of the conditional distributionally robust functional is defined as  a $\G$-measurable variable $\R_{|\G}(Z)\colon\O\to\eR$
satisfying the following properties:
\begin{enumerate}\item [{\rm (i)}] for  every $Q\in \cM$  the inequality    $\essinf \bbe_{Q|\G}[Z]\preceq \R_{|\G}(Z)$  holds,
  \item [{\rm (ii)}] if $Y\colon\O\to\eR$ is a $\G$-measurable variable such that  $\essinf \bbe_{Q|\G}[Z]\preceq Y$ holds   for every $Q\in \cM$, then  $\R_{|\G}(Z)\preceq Y$.
\end{enumerate}
We  use  the notation
	\begin{equation}\label{condfun}
		\R_{|\G}(Z)=\ess_{Q\in \cM} \bbe_{Q|\G}[Z]
	\end{equation}	for the \emph{conditional distributionally robust}  functional.
\end{definition}
Note that if $\G=\{\emptyset,\O\}$ is trivial, then  $\R_{|\G}$ coincides with the initial distributionally robust functional~$\R$ in~\eqref{fun-2}.

\begin{proposition}\label{cdrineq}
	For all $Z\in\Z$ it holds that \[\R(Z) \le \R\big( \R_{|\G}(Z)\big).\]
\end{proposition}
\begin{proof}
	For every $Q\in\cM$ we have from the tower property of the expectation that
	\[\E_Q\R_{|\G}(Z) \ge  \E_Q\E_{Q|\G}[Z]= \E_Q[Z].\] The assertion follows by passing to the supremum with respect to $Q\in\cM$.
\end{proof}

It is not difficult to verify that  $\R_{|\G}$ satisfies the respective counterparts of the axioms~\ref{enu:A1},~\ref{enu:A2} and~\ref{enu:A4} specified in Section~\ref{sec:basic}.
The conditional  counterpart of the  translation equivariance axiom~\ref{enu:A3} is: $\R_{|\G}(Z+Y)=\R_{|\G}(Z)+Y$ for any $Z\in \Z$  and $\G$-measurable $Y\in \Z$. If there is a  set $A\in \F$  such that $P(A)>0$ and $Q(A)=0$ for \emph{every} $Q\in \cM$, then $\R_{|\G}$ would  not satisfy the translation equivariance axiom (see the discussion of Remark~\ref{ex-sigmalg}).
 If $P\in \cM$ then this cannot happen   and  the translation equivariance axiom follows.

\subsection{Conditioning on strictly monotone risk functionals}
It was already mentioned that the conditional distributionally robust
functional~$\R_{|\G}$ is monotone, i.e., if  $Z'\succeq Z$, then   $\R_{|\G}(Z')\succeq \R_{|\G}(Z)$.
For verification of certain interchangeability properties of minimization operators and risk measures, which are  used for deriving the dynamic equations, there is a need for a stronger property of strict monotonicity \rev{(cf.\ \cite{ShapiroInterchangeability}).}
When $Z'\succeq Z$ and $Z'\ne Z$ we write this as $Z'\succ Z$, which  means that $P$-almost surely   $Z'\ge  Z$ and $P\{Z'>Z\}>0$.

\begin{definition}[Strict monotonicity]\label{def-strict}
It is said that $\R_{|\G}$ is \emph{strictly monotone} if for any  random variables $Z^\prime$, $Z$ such that $\R_{|\G}(Z')$ and  $\R_{|\G}(Z)$ are well-defined the following implication holds:
  $Z'\succ Z$ implies  $\R_{|\G}(Z')\succ \R_{|\G}(Z)$.
\end{definition}

If $\G=\{\emptyset,\O\}$ is trivial, then $\R_{|\G}=\R$ and  Definition~\ref{def-strict} reduces  to the definition of   strict monotonicity for the distributionally robust functional $\R$.
The natural question is:
``What is a relation between the strict monotonicity of $\R$ and its conditional counterpart?''
\begin{proposition}
	Suppose that $\R$ is strictly monotone. Then for any $A\in \F$ such that $P(A)>0$ there is $\e>0$ such that $Q (A)\ge \e$ for every $Q\in \cM$.
\end{proposition}
\begin{proof}
	Let \[\cA:=\{\zeta\in \Z^*\colon\zeta=dQ/dP,\;Q\in \cM\}\] be the corresponding (convex, bounded, weakly$^*$ closed)  dual set of density functions.
	We have that $\R$ is strictly monotone iff (cf.\ \cite{ShapiroInterchangeability})
	\begin{equation}\label{condstric-set}
		\int_A \zeta (\w) P(d\w)>  0\;\text{for every}\;  \zeta\in \cA\;\text{ and any}\;  A\in \F\; \text{such that} \;  P(A)>0.
	\end{equation}
	Note that since the set~$\cA$ is weakly$^*$  compact, the minimum of $\int_A \zeta\,dP$ over $\zeta \in \cA$ is attained. Therefore, for  the  respective  set  $\cM$   condition~\eqref{condstric-set} can be restated as announced.
\end{proof}

The property~\eqref{condstric-set}     implies  that $P$ is absolutely continuous with respect to every $Q\in \cM$. Thus, if $\R$ is strictly monotone, then
$\essinf  \bbe_{Q|\G}[Z]$ coincides with
    $\bbe_{Q|\G}[Z]$ for every $Q\in \cM$.

\begin{remark}\label{rem-nonstrict}
	Suppose that $\G=\F$, and  $P$-almost surely $\R_{|\F}(Z)=Z$ for every $Z\in \Z$. Then $\R_{|\F}(Z')\ne \R_{|\F}(Z)$ whenever   $Z'\ne Z$, and hence $\R_{|\F}$    is strictly monotone by the definition. On the other hand,     $\R$  can be not strictly monotone, i.e., it can happen that there is a measurable  set $A$ such that $Q(A)=0$ for some $Q\in \cM$ while $P(A)>0$.
	This indicates  that the conditional functional $\R_{|\G}$ can be strictly monotone even if $\R$ is not.
\end{remark}

\begin{proposition}
\label{pr-condstrmon}
If the distributionally robust functional $\R$ is strictly monotone, then the respective  conditional distributionally robust  functional  $\R_{|\G}$ is strictly monotone.
\end{proposition}

\begin{proof}
Suppose that  $\R$ is  strictly monotone. Let
  $Z'$, $Z\in \Z$ be  such that $Z'\succ Z$.
   It follows that $Z'\succeq Z+Y$, where $Y=\gamma\,\ind_A$
    for some $\gamma>0$ and $A\in \F$ such that $P(A)>0$.  We need to show that  $\R_{|\G}(Z')\ne \R_{|\G}(Z)$. We argue by a contradiction, so suppose that $\R_{|\G}(Z')= \R_{|\G}(Z)$.
    By Proposition~\ref{pr-essexist}   there is a sequence $\{Q_n\}\subset \cM$ such that $\R_{|\G}[Z]=\sup_{n\in \bbn} \bbe_{Q_n|\G}[Z]$.
We have that
\[
\R_{|\G}(Z')\succeq \sup_{n\in \bbn}\bbe_{Q_n|\G}[Z'] =  \sup_{n\in \bbn}\left(\bbe_{Q_n|\G}[Z]+  \bbe_{Q_n|\G}[Y]\right)\succeq  \sup_{n\in \bbn}\bbe_{Q_n|\G}[Z].
\]
It follows that $\bbe_{Q_n|\G}[Y]$    converges a.s.\ to $0$ as $n\to\infty$.
On the other hand,  $$\bbe\left\{\bbe_{Q_n|\G}[Y]\right\}=\bbe[Y]=\gamma\,Q_n(A).$$
Moreover, by~\eqref{condstric-set} we have that $Q_n(A)\ge \e$ for some $\e>0$ and all $n\in \bbn$.  It follows that $\bbe_{Q_n|\G}[Y]$  cannot  converge a.s.\ to $0$, which gives the required  contradiction. This completes the proof   that $\R_{|\G}$ is strictly monotone.
\end{proof}

As it was pointed in Remark~\ref{rem-nonstrict}, the converse implication: ``if $\R_{|\G}$ is strictly monotone, then $\R$ is strictly monotone'' does not hold.

\subsection{The reference measure}\label{rem-noref}
The preceding subsections require that there exists a reference probability measure. However, there are natural examples where there is no  a priori defined  reference probability measure. That is, the measures in~$\cM$ are not necessarily defined  with respect to a reference measure and indeed, there are situations where a priori  reference probability measure does not  exist.
This section describes situations where it is possible to construct a reference probability measure~$P$ on $(\O,\F)$.
We say that a probability measure~$P$  is a \emph{reference measure} for $\cM$ if every $Q\in\cM$ is absolutely continuous with respect to $P$.
The following construction reveals the smallest of all reference measures  dominating all measures in $\cM$.

 Consider
\begin{equation}\label{eq:3}
	\mu(A):= \sup\limits_{Q_i\in \cM}\left\{\sum_{i\in I}  Q_i(A_i)\colon A= \bigcup_{i\in I} A_i,\; |I|<\infty \right\},\quad A\in \F.
\end{equation}
The supremum in~\eqref{eq:3} is taken  among all finite partitions
	$\{A_i\}_{i\in I}$  of the set $A$ and  $\{Q_i\colon i\in I\}\subset\cM$.
In particular, if  the set $\O=\{\w_1,\dots\}$ is countable equipped  with the $\sigma$\nobreakdash-algebra~$\F$ of all its subsets, 
then
\begin{equation}\label{eq-countab}
	\mu(\{\w_i\})=\sup_{Q\in \cM}Q(\{\w_i\}).
\end{equation}

\begin{proposition}[Reference measure]\label{prop:EovM}
  The set valued  function $\mu(\cdot)$, defined in~\eqref{eq:3},
is the smallest measure such that
\begin{equation}\label{smallest}
	Q(A)\le\mu(A),\quad\forall Q\in\cM, A\in\F.
\end{equation}
\end{proposition}

\begin{proof}
	It is evident that $\mu(A)\ge 0$ and $\mu(\emptyset)=0$. To demonstrate that~\eqref{eq:3} is a measure it remains to demonstrate its $\sigma$-additivity.
	To this end consider a  set $B\in \F$ and let $\{B_j\}_{j\in \bbn}$ be its countable partition. By the definition of $\mu$ for any $\e>0$ there is a finite partition $\{A_i\}_{i\in I}$  of $B$ such that
\begin{equation}\label{eq:4}
		\mu(B) < \sum_{i\in I}\mu(A_i) +\e.
	\end{equation}
	Then
	\[	\mu(B)-\varepsilon<\sum_{i\in I}\mu(A_i)
		=\sum_{j\in \bbn} \sum_{i\in I}  Q_i(A_i\cap B_j)
		\le\sum_{j\in \bbn} \mu(B_j).
\]

	For the converse inequality, let $\{A_{ji}\}_{i\in I_j}$ be a partition of $B_j$, $j\in \bbn$,  such that
	\begin{equation}\label{eq:2}
		\mu(B_j)<\sum_{i\in I_j} Q_{i}(A_{ji})+\e\, 2^{-j}.
	\end{equation}
By summing~\eqref{eq:2} for $j\in \bbn$,  we obtain   that  $\sum_{j\in \bbn} \mu(B_j) < \mu(B) + \e$.  This demonstrates that $\mu(B)=\sum_{j\in \bbn} \mu(B_j)$, and hence $\mu$ is a measure.

	Finally, let $\nu$ be a measure such that $\nu(B)\ge Q(B)$ for all $Q\in\cM$. Let the partition  be chosen as in~\eqref{eq:4}.
	Then
	\[	\nu(B)=\sum_{i\in I}  \nu(A_i)\ge\sum_{i\in I}  Q_i(A_i)> \mu(B)-\varepsilon.\]
	It follows that $\nu\ge\mu$ and $\mu$ thus is the smallest of all measures dominating the measures in~$\cM$.
\end{proof}

\begin{remark}
We have by~\eqref{smallest} that $\mu$ dominates every measure $Q\in \cM$, i.e., $Q\ll \mu$ for all $Q\in\cM$.
The measure $\mu$   is not necessarily a probability measure. If $\mu(\Omega)< \infty$, then it can be normalized by considering  the probability measure
$\tilde P(A):=  \mu(A)/\mu(\Omega)$, $A\in \F$.
\end{remark}

 \subsection{Conditioning risk functionals on countable partitions}
As it was pointed above, in some settings there is no a priori  defined reference probability measure on $(\O,\F)$. This motivates to consider the following construction for conditioned risk functionals on countably generated sigma fields.

Let $\{\Upsilon_i\}_{i\in I}$ be a {\em countable} partition of the set $\O$ (cf.\ Definition~\ref{def:partition}).  This partition generates a subalgebra $\G:= \sigma(\Upsilon_i\colon i\in I)$ of $\F$,  consisting of  the empty set and  the sets   $A=\cup_{j\in \J} \Upsilon_j$ taken over all index sets   $\J\subset I$.
Then $X\colon\O\to\bbr$ is $\G$-measurable iff $X(\w)$ is constant on every $\Upsilon_i$, $i\in I$.
The conditional expectation $\bbe_{Q|\G}[X]$  is a
$\G$-measurable  function  $\bbe_{Q|\G}[X]\colon\O\to \bbr$
taking the following  value at $\w\in \Upsilon_i$, $i\in I$, such that $Q(\Upsilon_i)\ne 0$,
\begin{align}\label{condexpect}
	\bbe_{Q|\G}[X](\w)&=\frac{1}{Q(\Upsilon_i)}\int_{\Upsilon_i} X(\w') Q(d\w') \\
	\label{condexpect-2}
	&=\frac{1}{Q(\Upsilon_i)}\sum_{\w'\in \Upsilon_i} Q(\{\w'\})\,X(\w').
\end{align}
When the  $\sigma$\nobreakdash-subalgebra~$\G$ is  generated by the countable partition, formula~\eqref{condexpect}   defines  a version of the corresponding conditional expectation. Again any
two  versions of  $\bbe_{Q|\G}[X]$ can have different values on such $\Upsilon_i$ that $Q(\Upsilon_i)=0$.

As above, we can now define a version of   $\R_{|\G}(Z)$ as a $\G$-measurable variable taking value
\begin{eqnarray}
\label{condcount}
	\R_{|\G}(Z)(\w)&:=& \sup_{Q\in \cM,\;Q(\Upsilon_i)\ne 0} \frac{1}{Q(\Upsilon_i)}\int_{\Upsilon_i} Z(\w') Q(d\w')\\
\nonumber
&=& \sup_{Q\in \cM,\;Q(\Upsilon_i)\ne 0} \bbe_{Q|\G}[Z](\w)
\end{eqnarray}
at $\w\in \Upsilon_i$ such that   $Q(\Upsilon_i)\ne 0$  for at least one $Q\in \cM$.

  We have here that   $\essinf \bbe_{Q|\G} [Z](\w)=-\infty$ for such  $\w\in\Upsilon _i$   that $Q(\Upsilon_i)= 0$. Therefore, in the case of countable partition, formula~\eqref{condcount} gives a closed form expression for the $\esssup_{Q\in\cM}\bbe_{Q|\G} [Z]$ described in Definition~\ref{def-sup-1}.
  Let $\bar{\Upsilon}$  be the union of the sets $\Upsilon_i$ such that  $Q(\Upsilon_i)\ne 0$ for at least one $Q\in \cM$.
  Then  a version of  the conditional functional $\R_{|\G}(Z)(\w)$ is uniquely defined for $\w\in \bar{\Upsilon}$, and can be arbitrary on $\O\setminus \bar{\Upsilon}$. If the set $\bar{\Upsilon}$ is strictly smaller than $\O$, i.e., there is $i\in I$ such that $Q(\Upsilon_i)= 0$ for all $Q\in \cM$, then the corresponding  translation equivariance property does not hold. In fact, it is not difficult to verify that this condition is also sufficient for the translation equivariance. That is, the following lemma holds.
\begin{lemma}
	In the considered framework of countable partition, the conditional translation equivariance property  holds iff for every $i\in I$ there is $Q\in \cM$ such that $Q(\Upsilon_i)> 0$.
 \end{lemma}

\paragraph{Framework of $C(\O)$-setting}
Let $\O$ be a compact metric space and  $\R$ be  defined in the form~\eqref{fun-2} with $\cM\subset  C(\O)^*$ being  a nonempty convex  set of probability measures. We can assume that $\cM$ is closed in the weak$^*$ topology of $C(\O)^*$.
Let~$\G$  be a  $\sigma$\nobreakdash-subalgebra of~$\F$ generated by a countable  partition $\{\Upsilon_i\}_{i\in I}$ of the set $\O$. Then  $  \R_{|\G}$ can be written  in the form~\eqref{condcount}. In particular,  this framework is relevant for the moment constraints setting (we will discuss this below).  Here   the functional~$\R$ is strictly monotone iff the following condition holds (cf.\ \cite{ShapiroInterchangeability}):
\begin{equation}\label{condstric-2}
	Q (A)>  0\;\text{for every}\;  Q\in \cM\;\text{ and any open set }\;   A\subset \O.
\end{equation}
Consequently, if $\R$ is strictly monotone, then $\essinf  \bbe_{Q|\G}[Z]$ coincides with
    $\bbe_{Q|\G}[Z]$ for every $Z\in \Z$.

Let  $Z$, $Z'\in C(\O)$  be such that $Z'\succ Z$. Since $Z$, $Z'\colon\O\to\bbr$ are continuous functions, this implies that there is an open set $A\subset \O$ such that $Z'(\w) > Z(\w)$ for all $\w\in A$.
By using~\eqref{condstric-2} it is possible to show in a way  similar   to Proposition~\ref{pr-condstrmon},  that in this framework  strict monotonicity of   $\R$  implies   strict  monotonicity of the corresponding conditional functional  $\R_{|\G}$.

\begin{example}\label{ex-moment}
{\rm
	Let $\O\subset \bbr^d$ be a convex compact  set  with a nonempty interior,  equipped with its Borel $\sigma$\nobreakdash-algebra~$\F$. Suppose that the set $\cM$ consists of a family of probability measures $Q$ on $(\O,\F)$  such that $|\supp (Q)|\le n$, i.e., the  support of each  $Q\in \cM$ has no more than $n$ points. This setting is relevant when the ambiguity set is defined by a finite number $m$  of moment constraints.
	By the Richter--Rogosinski Theorem  in that  case it suffices to consider probability measures supported on no more than $m+1$ points. Let $\G$ be the $\sigma$\nobreakdash-algebra generated by a
	countable  partition $\{\Upsilon_i\}_{i\in I}$ of  $\O$.
	Then formula~\eqref{condcount} can be applied using  the respective conditional expectation in the form~\eqref{condexpect-2}. By~\eqref{condstric-2} the corresponding distributionally robust functional $\R$ cannot be strictly monotone in this setting. On the other hand, its conditional counterpart  $\R_{|G}$ is strictly monotone if  $Q(\Upsilon_i)>0$ for every $Q\in \cM$ and $i\in I$. Of course this could happen only if $|I|\le n$.
}
\end{example}

\section{Composite  distributionally robust functionals}
\label{sec-nested}

The results presented in this section are motivated by \cite{Shapiro2015}.
Unless stated otherwise we assume here  the $L_p$-setting of Section~\ref{ex-lp}. We also assume that the  translation equivariance
axiom  holds for the conditional distributionally robust  functional $\R_{|\G}$ for  every subalgebra~$\G$, in particular
$\R_{|\F}(Z)=Z$ for every $Z\in \Z$.
Let $\F_1\subset \cdots\subset \F_T$ be a filtration (sequence of $\sigma$\nobreakdash-algebras)  on  $(\O,\F)$, with $\F_T=\F$ and  $\F_1=\{\emptyset, \O\}$ being trivial.  By the tower property  of the conditional expectation  we have for $Q\in \cM$ that
\begin{equation}\label{nest-1}
	\bbe_Q[Z]=\bbe_{Q|\F_1}\Big[\bbe_{Q|\F_2} \big[\cdots \bbe_{Q|\F_{T-1}}[\bbe_{Q|\F_T}[Z]]\big]\Big].
\end{equation}
Note that since $\F_1$ is trivial,  $\bbe_{Q|\F_1}$ is just the unconditional expectation $\bbe_{Q}$, we write it in that form for uniformity of the notation.

Suppose that $\ess_{Q\in \cM}\bbe_{Q|\F_t}(Z)$ is well-defined for every $Z\in \Z$ and  $t=1, \dots,T$. It follows  that
\begin{align}\label{set-r1}
 	\R(Z)&=\sup_{Q\in \cM}\bbe_{Q|\F_1}\Big [\bbe_{Q|\F_2} \big[\cdots \bbe_{Q|\F_{T-1}}[ \bbe_{Q|\F_T}[Z]]\big]\Big]\\
 \label{set-r2}
	&\le \sup_{\Q_1\in \cM}\bbe_{\Q_1|\F_1}\Big [ \cdots\ess_{\Q_{T-1}\in \cM} \bbe_{\Q_{T-1}|\F_{T-1}}\big[ \ess_{\Q_T\in \cM} \bbe_{\Q_T|\F_T}[Z]\big]\Big].
\end{align}
That is (cf.\ Proposition~\ref{cdrineq})
\begin{equation}\label{nest-2}
	\R(Z)\le \cR (Z),
\end{equation}
where
\begin{equation}\label{nestdec-1}
	 \cR (Z)  := \sup_{\Q_1\in \cM}\bbe_{\Q_1|\F_1}\Big [  \cdots\ess_{\Q_{T-1}\in \cM} \bbe_{\Q_{T-1}|\F_{T-1}}\big[ \ess_{\Q_T\in \cM} \bbe_{\Q_T|\F_T}[Z]\big]\Big].
\end{equation}
This can be also written as
\begin{equation}\label{nestdec-2}
 \cR (Z) =\R_{|\F_1}\Big (\R_{|\F_2} \big(\cdots \R_{|\F_{T-1}}(\R_{|\F_T}(Z))\big)\Big),
\end{equation}
with $\R_{|\F_t}$  being the respective conditional distributionally robust  functionals. Note that $\R_{|\F_1}=\R$ since $\F_1$ is trivial.

\begin{remark}
	Suppose that  the conditional functionals $\R_{|\F_t}$ in the right-hand side of~\eqref{nestdec-2} are defined as the \emph{nested} counterparts of the corresponding law invariant coherent  risk measures.  Formula~\eqref{nestdec-2} still defines a (possibly not law invariant) coherent risk measure $\cR$. For example $\R_{|\F_t}$ can be the nested $\avr_{\alpha|\F_t}$. In that case it  can happen that the inequality~\eqref{nest-2} does not  hold (see, e.g., \cite{Pichler2012a, Pichler2012b}).
\end{remark}

Note that since $\F_T=\F$ it follows that $Q$-almost surely  $\bbe_{Q|\F_T}[Z]=Z$. However, as it was discussed in the previous section, this does not imply that $\bbe_{Q|\F_T}[Z]=Z$ in the $P$-a.s.\ sense unless $P$ is absolutely continuous with respect to $Q$. Anyway, because of the assumption $\R_{|\F}(Z)=Z$, we have that $\R_{|\F_T}(Z)=Z$ for every $Z\in \Z$.
We refer to $\cR$ as the \emph{composite} (distributionally robust)  functional associated with the filtration $\cF=\{\F_1, \dots,\F_T\}$.

Suppose that $\cR (Z)$ is finite valued for every $Z\in \Z$. Since the conditional functionals inherit the respective properties of subadditivity, monotonicity and positive homogeneity, it follows that $\cR \colon\Z\to\bbr$ satisfies the respective axioms~\ref{enu:A1},~\ref{enu:A2} and~\ref{enu:A4}.
The translation equivariance axiom is also assumed to hold. Since $\cR$ is finite valued, it follows that it is continuous and has the following dual representation (cf.\ \cite[Theorem~6.7]{RuszczynskiShapiro2009}).
\begin{proposition}\label{pr-dualrepnest}
	Suppose that the composite functional $\cR \colon\Z\to\bbr$ is finite valued and the translation equivariance axiom  holds for every subalgebra of $\F$.
	Then there exists a convex, bounded, weakly$^*$ closed set  $\widehat{\cA}\subset \Z^*$ of density functions such that
	\begin{equation}\label{eq-dualnested}
		\cR(Z)=\sup_{\zeta\in \widehat{\cA}}\int_\O Z(\w)\zeta(\w) P(d\w),\quad Z\in \Z.
	\end{equation}
\end{proposition}

We can also write~\eqref{eq-dualnested} in the form
\begin{equation}\label{eq-dualnest-2}
	\cR(Z)=\sup_{Q\in \widehat{\cM}}\bbe_Q[Z],\;\;Z\in \Z,
\end{equation}
where  $\widehat{\cM}$ is the set of probability measures $Q$ absolutely continuous with respect to the reference probability measure $P$ such that $dQ/dP\in \widehat{\cA}$, that is
\begin{equation}\label{eq-dualn-3}
	\widehat{\cM}=\big\{Q\in \cP\colon dQ/dP\in \widehat{\cA}\big\}.
\end{equation}
Compared with the representation~\eqref{fun-1} of $\R$, we have by~\eqref{nest-2} that $\cM\subset \widehat{\cM}$.
Moreover, the inclusion is strict, i.e. $\widehat{\cM}\ne \cM$, if the inequality~\eqref{nest-2} is strict for some $Z\in \Z$.

\rev{
The set $\widehat{\cM}$,  in the representation~\eqref {eq-dualnest-2},  is derived in an abstract way;  its constructive description  could be  complicated even in the rectangular case discussed below (see Example~\ref{ex-twost}).
}

\begin{remark}\label{rem-confun}
	A similar analysis can be applied to the $C(\O)$-setting. Since the space $C(\O)$ is a Banach lattice, we have that if a functional defined on $C(\O)$  is real (finite) valued convex and monotone, then it is continuous in the norm topology of $C(\O)$ (cf.\ \cite[Theorem~7.91]{RuszczynskiShapiro2009}). It follows that if the composite  functional $\cR$, of the form~\eqref{nestdec-2}, is well-defined and finite valued, then it has the dual representation~\eqref{eq-dualnest-2} for some set $\widehat{\cM}\subset C(\O)^*$ of probability measures.

	In the $\bB(\O)$-setting the situation is more delicate. Even if the corresponding composite  functional is well-defined and finite valued, in order to ensure its dual representation~\eqref{eq-dualnested} it should be verified that it is lower semi-continuous with respect to the considered paired topologies.
\end{remark}

\subsection {The nested functional}\label{sec-nestrect}
As it was discussed in Section~\ref{sec-condrisk} the nested and the `$\ess$'  approaches could lead to different definitions of conditional functionals.  In order to reconcile these two approaches   we can proceed as follows.
Let $\Xi_t\subset \bbr^{d_t}$ be  a nonempty, closed set equipped with its Borel $\sigma$\nobreakdash-algebra~$\B_t$, $t=1, \dots,T$, and $\B=\B_1\otimes\cdots\otimes \B_T$
be the Borel $\sigma$\nobreakdash-algebra on the set $\Xi:= \Xi_1\times\cdots\times \Xi_T$. We deal in this section with the measurable space $(\Xi,\B)$ and  the corresponding
space $\Z$ of random variables $\Xi\to\bbr$. For $t=2,\dots,T$,
let  $\M_{t}^{\xi_{[t-1]}}$ be a  nonempty    set of (marginal) probability  distributions on $(\Xi_t,\B_t)$, depending  on the history $\xi_{[t-1]}\in \Xi_1\times \cdots\times \Xi_{t-1}$, and $\M_1$ be a set  of probability distributions on $(\Xi_1,\B_1)$.

Consider a variable $Z\in \Z$,
and going backward in time for $t=T,\dots,2$, starting with $Z_T=Z$,
define
\begin{equation}\label{int-recnest-1}
	Z_{t-1}(\xi_{[t-1]}):= \ess_{Q_t\in\M_{t}^{\xi_{[t-1]}}}
\left\{\bbe_{Q_t|\xi_{[t-1]}}[Z_t]= \int_{\Xi_t}
	Z_t(\xi_{[t-1]},\xi_t)\,Q_t(d\xi_t)\right\}.
\end{equation}
 The corresponding \emph{nested} functional is defined as
 \begin{equation}\label{nest-2a}
	\cR(Z):= \sup_{Q_1\in \M_1} \bbe_{Q_1}\left[ Z_1(\xi_1)\right].
\end{equation}
In the construction~\eqref{int-recnest-1}--\eqref{nest-2a}
there is no naturally defined ambiguity set $\cM$ and the associated distributionally robust   functional $\R$ is obtained by the construction~\eqref{set-r2}.
On the other hand, if the above  functional  $\cR$ is finite valued on the space $\Z$ then it can be represented in the form~\eqref{eq-dualnest-2}. The corresponding set $\widehat{\cM}$ of probability distributions could be quite complicated.

In a finite dimensional setting we can think about data process as the corresponding scenario  tree,  with a node at stage $t$ representing the history $\xi_{[t]}:=(\xi_1,\dots,\xi_t)$ of the   process. Then at every node $\xi_{[t]}$ is defined the corresponding set $\M_{t}^{\xi_{[t-1]}}$ of conditional probability distributions of the child nodes of $\xi_{[t]}$. With the set $\M_{t}^{\xi_{[t-1]}}$ is associated   the respective coherent  risk measure defined  on the space of child nodes.
The composition of these conditional risk measures defines the corresponding nested risk measure. For such  detail construction we can refer to \cite[Section~6.8.1]{RuszczynskiShapiro2009}.
In the general setting of continuous distributions a rigorous formulation of the right-hand side of~\eqref{int-recnest-1} could be quite delicate and is beyond the scope of this paper. \rev{We can refer to \cite[Section~5]{Kallenberg2002Foundations} for an introduction of disintegration with mathematical rigor and for an appropriate  discussion.}


\medskip
Next we discuss the  rectangular setting where the conditional distributionally robust approach, associated with an ambiguity set,   is  equivalent to the nested construction.

\subsection{Rectangular setting}
\label{sec-rec}

The nested construction~\eqref{int-recnest-1} is simplified considerably if the sets of marginal distributions do not depend on the history of the data process.
That is, for $t=1, \dots,T$,
let $\M_t$ be a nonempty  set of probability measures on $(\Xi_t,\B_t)$ and consider the corresponding set
\begin{equation}\label{eq-rect}
	\cM    :=\left\{Q=Q_1\times\cdots\times Q_T\colon Q_t\in \M_t,\;t=1, \dots,T\right\}
\end{equation}
of probability measures on $(\Xi,\B)$.
We refer to this setting as \emph{rectangular}.
The vector $\xi_t\in \Xi_t$ can be viewed as an element of the measurable space $(\Xi_t,\B_t)$ or as a random vector having a considered probability distribution (measure) on $(\Xi_t,\B_t)$.
The sets $\M_t$ represent   respective marginal distributions and the rectangular setting can be considered as a distributionally robust counterpart  of the stagewise  independence condition.
If we view $\xi_1, \dots,\xi_T$  as a random process having distribution $Q\in \cM$, then~\eqref{eq-rect}   means that the random
vectors $\xi_1, \dots,\xi_T$, are mutually independent with respective marginal distributions $Q_t\in \M_t$.

The corresponding variables, defined in~\eqref{int-recnest-1}, can be written here as
\begin{equation}\label{int-3}
	Z_{t-1}(\xi_{[t-1]}):= \sup_{Q_t\in \M_t}\left\{\bbe_{Q_t|\xi_{[t-1]}}[Z_t]=
	\int_{\Xi_t}
	Z_t(\xi_{[t-1]},\xi_t)\,Q_t(d\xi_t)\right\}.
\end{equation}
The  supremum  in~\eqref{int-3}   is taken with respect to the respective marginal distributions.

The rectangular setting constitutes a notable, special case.
\begin{remark}[Equivalence of the nested and the conditional approach]\label{rem:1}
	In the considered rectangular  setting the composite  functional $\cR$  does \emph{not} depend on whether the nested or  conditional distributionally robust  approach is used.
\end{remark}

\begin{example}[Average Value-at-Risk]
\label{ex-rectavr}
{\rm
Suppose that for every $t=1,\dots,T$,  the set $\M_t$ of marginal distributions consists of probability measures absolutely continuous with respect to a reference distribution~$P_t$ on $(\Xi_t, \B_t)$  and the corresponding densities correspond to the $\avr_{\alpha_t}$ risk measure taken with respect to $P_t$. Then for $t=T,\dots,2$ and $Z_T=Z$ we have
\begin{equation}\label{recavr}
 Z_{t-1}(\xi_{[t-1]})=\avr_{\alpha_t}\big(Z_t(\xi_{[t-1]},\xi_t)\big),
\end{equation}
where for given $\xi_{[t-1]}$
the  Average Value-at-Risk of $Z_t(\xi_{[t-1]},\cdot)$ is
computed with respect to the distribution $P_t$ of $\xi_t$.  Finally,
$\cR(Z)=\avr_{\alpha_1}(Z_1)$.
}
\end{example}

When $\xi_{[T]}$ is viewed as random, the supremum should be replaced by the essential supremum. Consequently, the composite  functional $\cR$ can be written as
\begin{equation}\label{nest-1aa}
	\cR(Z)= \sup_{Q_1\in \M_1} \bbe_{Q_1}\left[\ess_{Q_2\in \M_2}\bbe_{Q_2|\xi_{[1]}}\Big [\,\cdots\,  \ess_{Q_T\in \M_T}\bbe_{Q_T|\xi_{[T-1]}}[Z]\Big]\right].
\end{equation}
 On the other hand the functional $\R(Z)$  is
\begin{equation}\label{nest-3}
	\R(Z)= \sup_{Q_1\in \M_1, \dots,Q_T\in \M_T}\int_{\Xi}
	Z(\xi_1, \dots,\xi_T)\,Q_1(d\xi_1)\dots Q_T(d\xi_T),
\end{equation}
$\Xi=\Xi_1\times\cdots\times \Xi_T$.
As before the inequality~\eqref{nest-2} and the representation~\eqref{eq-dualnest-2} hold. The set $\widehat{\cM}$ is different from the set $\cM$ and does not have the rectangular structure of the form~\eqref{eq-rect}.
The equality $\R(Z)=\cR(Z)$, for some $Z\in \Z$,  can happen only  in rather exceptional cases even in this rectangular setting.

\begin{remark}
\label{rem-orderint}
The integral in the right-hand side of~\eqref{nest-3} does not depend on the order of integration, i.e., for any   permutation  $\{i_1,\dots,i_T\}$ of the set $\{1,\dots,T\}$, we can write
\begin{equation}\label{nest-perm}
	\R(Z)= \sup_{Q_{i_1}\in \M_{i_1}, \dots,Q_{i_T}\in \M_{i_T}}\int
	Z(\xi_1, \dots,\xi_T)\,Q_{i_1}(d\xi_{i_1})\dots Q_{i_T}(d\xi_{i_T}).
\end{equation}
On the other hand in the definition~\eqref{nest-1aa} of the composite functional $\cR$ the order is important.
\end{remark}

\subsection{Difference of the distributionally robust and the composite functional}
Remark~\ref{rem:1} above addresses a special risk functional which can be employed in both, the nested and the conditional settings. In order to see how the composite   functional~$\cR$ differs from~$\R$  let us consider (in the rectangular framework) the following example  of two stage case.

\begin{example}
\label{ex-twost}
{\rm
For  $T=2$ \rev{and the rectangular setting}  consider the corresponding composite functional
\begin{equation}\label{rec9}
	\cR(Z)= \sup_{Q_1\in \M_1} \bbe_{Q_1}\left[\sup_{Q_2\in \M_2}\bbe_{Q_2|\xi_1}[Z]\right].
\end{equation}
Under suitable  regularity conditions \rev{(e.g.,  \cite[Theorem 14.60]{WetsRockafellar97})}, we can interchange the supremum  and integral operators to write
\begin{equation}\label{rec10}
	\bbe_{Q_1}\left[\sup_{Q_2\in \M_2}\bbe_{Q_2|\xi_1}[Z]\right]=
	\sup_{Q^{(\cdot)}_2\in \cV}
	\bbe_{Q_1}\left[\bbe_{Q^{\xi_1}_2|\xi_1}[Z]\right].
\end{equation}
The maximum in the right-hand side of~\eqref{rec10} is with respect to  family $\cV$ of  mappings $\xi_1\mapsto Q^{\xi_1}_2$, from $\Xi_1$ to $\M_2$, such that   the integral
\begin{equation}\label{rec11}
	\bbe_{Q_1}\left[\bbe_{Q^{\xi_1}_2|\xi_1}[Z]\right]=
	\int_{\Xi_1}\left( \int_{\Xi_2} Z(\xi_1,\xi_2)\,Q^{\xi_1}_2(d\xi_2)\right)Q_1(d\xi_1)
\end{equation}
is well-defined. The notation $Q_2^{(\cdot)}$ emphasizes that the probability measure $Q^{\xi_1}_2$ in~\eqref{rec10} is a function of $\xi_1$. If we take the mapping
$\xi^1\mapsto Q^{\xi_1}_2$ in~\eqref{rec10} to be constant, i.e., $Q^{\xi_1}_2=Q_2$  for all  $\xi_1$, then the corresponding measure is $Q=Q_1\times Q_2$.

The right-hand side of~\eqref{rec11} defines a
functional $\phi\colon\Z\to\bbr$. This functional is
linear, monotone and $\phi(Z+a)=\phi(Z)+a$ for $Z\in \Z$ and $a\in \bbr$. Hence, there exists a probability measure $Q$ on $\Xi_1\times \Xi_2$, which depends on the mapping  $\xi^1\mapsto Q^{\xi_1}_2$ (and also on $Q_1$), such that
\begin{equation}\label{rec12}
	\bbe_{Q_1}\left[\bbe_{Q^{\xi_1}_2|\xi_1}[Z]\right]=	\bbe_{Q}[Z].
\end{equation}
For measurable sets $A_1\subset \Xi_1$ and $A_2\subset \Xi_2$, and
$Z(\xi_1,\xi_2):=
\ind_{A_1}(\xi_1)\ind_{A_2}(\xi_2)$, we have
\begin{equation}\label{rec13}
	\int_{A_1}\left( \int_{\A_2} Z(\xi_1,\xi_2)\, Q^{\xi_1}_2(d\xi_2)\right)Q_1(d\xi_1)
	=\int_{A_1}Q^{\xi_1}_2 (A_2)\, Q_1(d\xi_1)
\end{equation}
and hence, for
$A:= A_1\times A_2$,
\begin{equation}\label{rec14}
	Q(A)=\int_{A_1} Q^{\xi_1}_2 (A_2)\, Q_1(d\xi_1).
\end{equation}
The set  $\widehat{\cM}$  is obtained by taking  the union of such measures $Q$  over all mappings in  $\cV$ and $Q_1\in \M_1$,
and then taking the topological closure of the convex hull of that family.
Consequently,
\begin{align}\label{eq-hatem1}
	\cR(Z)&=\sup_{Q_1\in \M_1, Q_2^{(\cdot)}\in \cV}
	\int_{\Xi_1}  \int_{\Xi_2} Z(\xi_1,\xi_2)\,Q^{\xi_1}_2(d\xi_2) Q_1(d\xi_1)\\
	\label{eq-hatem2}
	&= \sup_{Q\in \widehat{\cM}}\bbe_Q[Z].
\end{align}

For example suppose that the set $\Xi_1$ is finite,  $|\Xi_1|=n$,  and  the  set  $\M_1$  is given by the convex hull of a finite  family $\{Q_1^1,\dots,Q_1^{m_1}\}$ of probability measures on $(\Xi_1,\B_1)$, and $\M_2$  is given by the convex hull of a finite  family $\{Q_2^1, \dots,Q_2^{m_2}\}$ of probability measures on $(\Xi_2,\B_2)$. Then the set $\cM$ is given by the convex hull of the family
\begin{equation}\label{family-1}
	\{Q_1^i\times Q_2^j\},  \;i=1, \dots,m_1, \;j=1, \dots,m_2.
\end{equation}

The set~$\widehat{\cM}$, on the other hand, is obtained by taking the convex hull of the family
\begin{equation}\label{family-2}
	\{Q_1^i\times Q_2^{j_k}\}, \;i=1, \dots,m_1, \;j_k\in \{1, \dots,m_2\}, \;k=1, \dots,n.
\end{equation}
 The number of elements (probability measures) in the family~\eqref{family-1} could be as large as $m_1\times  m_2$, while in the family~\eqref{family-2} it  could be as large as  $m_1\times  (m_2)^n$ (some of these probability measures could be equal to each other).
 Family~\eqref{family-1} is a subset of family~\eqref{family-2} obtained by taking indexes $j_k\equiv j$ constant (independent of $k$).
It could be noted  that $\cR(Z)=\R(Z)$ for some $Z\in \Z$ if the supremum  in~\eqref{eq-hatem1}  with respect to $Q_2^{\xi_1}$ does not depend on $\xi_1$.
}
\end{example}

In the considered rectangular framework it is possible  to consider cases without assuming existence of a reference probability measure.

\begin{paragraph} {Moment constraints} 
Suppose that for $t=1,\dots,T$, the set $\M_t$ consists of probability distributions $Q_t$ on $(\Xi_t,\B_t)$ such that
\begin{equation}\label{moment-1}
  \bbe_{Q_t}[\Psi_{ti}(\xi_t)]=b_{ti},\;i=1,\dots,m_t,
\end{equation}
for some measurable functions $\Psi_{ti}:\Xi_t\to \bbr$ and $b_{ti}\in \bbr$.
By duality, under mild regularity conditions,  the corresponding variable $Z_{t-1}(\xi_{[t-1]})$, defined in~\eqref{int-3}, is give by the optimal value of the following problem  (e.g., \cite[Section~6.7]{RuszczynskiShapiro2009})
\begin{equation}\label{moment-2}
\begin{array}{cll}
 \min\limits_{\lambda_t\in \bbr\times \bbr^{m_t}} &
  \lambda_{t0}+\sum_{i=1}^{m_t} b_{ti}\lambda_{ti}\\
  {\rm s.t.}&   \lambda_{t0}+\sum_{i=1}^{m_t} \lambda_{ti}\Psi_{ti}(\xi_t)\ge
 Z_t(\xi_{[t-1]},\xi_t),\;\xi_t\in \Xi_t.
 \end{array}
\end{equation}
 If every set $\Xi_t$ is compact (and hence the set $\Xi$ is compact) and  the function $Z_T=Z$ is continuous on the set $\Xi$, then every $Z_{t}$  is continuous on the set $\Xi_1\times\cdots\times \Xi_t$  (this can be shown by induction going backward in time and using, e.g., \cite[Theo\-rem~7.23]{RuszczynskiShapiro2009}). It is natural here to use   the framework of the $C(\Xi)$ setting (discussed in Section~\ref{ex-comeg}). Consequently, there is a set $\widehat{\cM}$ of probability measures on $(\Xi,\B)$ such that the corresponding composite functional $\cR$ is representable in the form~\eqref{eq-dualnest-2}. The set $\widehat{\cM}$ can be different from the corresponding  set $\cM$  (defined in~\eqref{eq-rect}).

\begin{example}\label{ex-mean}
{\rm
	Suppose that $\Xi_2:= [\alpha,\beta]$, for some $\alpha<\beta$, and the set  $\M_2$  consists of probability measures supported on the interval $[\alpha,\beta]$ and  having mean $\mu\in [\alpha,\beta]$. Suppose further that $Z(\xi_1,\xi_2)$ is convex continuous  in $\xi_2\in [\alpha,\beta]$. Then   the respective maximum over $Q_2\in \M_2$ is attained at the probability measure
	$Q^*_2=p\,\delta(\alpha)+(1-p)\delta(\beta)$
	supported on the end points of the interval $[\alpha,\beta]$, with $p:= (\beta-\mu)/(\beta-\alpha)$. For such $Z$ the respective maximum over $Q_2\in \M_2$ does not depend on $\xi_1$, and hence  $\cR (Z)=\R(Z)$. Note that the assumption that
	$Z(\xi_1,\xi_2)$ is convex in $\xi_2$ is essential here, and that  $\cR (Z)$ can be different from~$\R(Z)$ for general~$Z$.
}
\end{example}

\end{paragraph}

\subsection{Consequences for multistage stochastic optimization}
A typical setting in stochastic optimization considers robustifications at every new decision, i.e., based on the realized history of the entire process. In a scenario tree, this corresponds to a robustification at every node.
A convenient and widely used setting in this respect involves the Wasserstein metric and the nested distance. In what follows we elaborate this setting in the two-stage case first and   then proceed to  the more general multistage situation.

\subsubsection{Two-stage regularity}
The Wasserstein distance is a distance for probability measures. Given this distance~$d$ of measures, a typical ambiguity set is (called Wasserstein ambiguity set in \cite{EsfahaniKuhn, Kuhn2015})
\begin{equation}\label{eq:Ball}
	\cM:= \{Q\colon d(P,Q)\le r\},\end{equation}
where $P$ is some reference measure, often an empirical measure $P=\frac1n\sum_{i=1}^n\delta_{X_i}$ for independent observations $X_i$, $i=1,\dots,n$. The radius~$r$ of the Wasserstein ball~\eqref{eq:Ball} is given the interpretation of a confidence interval to collect all empirical measures with a probability of $95\,\%$, say.
\begin{definition}[Wasserstein distance]\label{def:Wasserstein}
	Let $P$ and $Q$ be probability measures on a Polish space $(\Xi, d)$. The Wasserstein distance of order $r\ge1$ is
	\begin{equation}\label{eq:Wasserstein}
		d_r(P, Q)^r= \inf \iint_{\Xi\times\Xi} d(\xi,\tilde \xi)^r\,\pi(d\xi,d\tilde \xi),
	\end{equation}
	where the infimum is among all bivariate measures~$\pi$ with marginals
	\begin{align}
		\pi(A\times \Xi) &=P(A)\text{ and} \label{eq:62}\\
		\pi(\Xi\times B) &=Q(B)    \label{eq:63}
	\end{align}
	\rev{for all Borel sets $A,B\in\mathcal B(\Xi)$.}
\end{definition}
The Wasserstein distance is of importance because of its duality relation, which is the content of the Kantorovich--Rubinstein theorem, cf.\ \cite{Villani2003} or \cite{Ambrosi2005}. Indeed, it holds that
\begin{equation}\label{eq:61}
	\left |\E_Q [Z] -\E_P [Z]\right | \le L_Z\cdot d_1(Q,P),
\end{equation}
where $L_Z$ is the Lipschitz constant of the random variable $Z$, i.e., $|Z(\xi)-Z(\tilde\xi)|\le L_Z\,d(\xi,\tilde\xi)$.
For Lipschitz continuous functions we find ourselves in the $C(\O)$ setting addressed in the introduction. Note that this setting allows addressing atomic and non-atomic probability measures simultaneously.

Based on~\eqref{eq:61} we have the following result in a two-stage setting.
\begin{theorem}
	Let $P$ be a reference measure and the ambiguity set
	\begin{equation}\label{eq:16}
		\cM=\{Q\colon d_1(P,Q)\le \epsilon\}
	\end{equation}
	collect all measures, which do note deviate by more than $\epsilon\ge0$ in Wasserstein distance from the reference measure~$P$.  Then the bound
	\begin{equation}\label{eq:18}
		\left |\R(Z)-\E_P[Z]\right |\le L_Z\cdot\epsilon
	\end{equation}
	holds for the risk functional $\R(Z)= \sup_{Q\in \cM} \E_Q[Z]$.
\end{theorem}
\begin{rev}
	The ambiguity set~\eqref{eq:16} is popular in applications (cf.\ \cite{EsfahaniKuhn}). For this reason we emphasize that the $C(\Omega)$ setting is important here, as the measures in~\eqref{eq:16} are non-dominated in the sense of Section~\ref{rem-noref}.
	Further, the bound~\eqref{eq:18} allows comparing the risk functional~$\R$ with the much simpler risk neutral expectation.
\end{rev}

\begin{rev}
\subsubsection{Multistage regularity and the nested distance}
	The nested distance builds on the Wasserstein distance, but respects the evolution of the process in addition.
\end{rev}
For the corresponding robustification in the multistage situation let $P\in\cP$ be the law of a reference process with disintegration
\[\E_P[Z]=\int_{A_1}\int_{A_2}\dots\int_{A_T} Z(\xi_1,\dots,\xi_T)\, P_T^{\xi_{[T-1]}}(d\xi_T)\dots P_2^{\xi_1}(d\xi_2)P_1(d\xi_1).\]
\begin{rev}
\begin{definition}[Nested distance, cf.\ Definition~\ref{def:Wasserstein} (Wasserstein distance)]
	For $P$ and $Q$ probability laws on a Polish space $(\Xi,d)$ with $\Xi:=\Xi_1\times\dots\times\Xi_T$, the \emph{nested distance} of order $r\ge1$ is
	\begin{equation*}
		\nd_r(P, Q)^r:= \inf \iint_{\Xi\times\Xi} d(\xi,\tilde \xi)^r\,\pi(d\xi,d\tilde \xi),
	\end{equation*}
	where the infimum is among all bivariate measures~$\pi$ on $\Xi\times\Xi$ with conditional marginals
	\begin{align*}
		\pi^{\xi_{[t]},\eta_{[t]}}(A\times \Xi) &=P^{\xi_{[t]}}(A)\text{ a.s.\ and} 
		\\
		\pi^{\xi_{[t]},\eta_{[t]}}(\Xi\times B) &=Q^{\eta_{[t]}}(B)\text{ a.s.}    
	\end{align*}
	for all Borel sets $A,B\in\mathcal B(\Xi)$ and times $t=1,\dots,T$.
\end{definition}
\end{rev}

Given some history $\xi_{[t-1]}$ consider now the specific ambiguity set $\M_t^{\xi_{[t-1]}}$ with
\begin{align}	\label{eq:17}
	Q_t^{\xi_{[t-1]}} & \in\M_t^{\xi_{[t-1]}}\\
	\text{ iff } &	d_1\big( Q_t^{\xi_{[t-1]}}, P_t^{\tilde\xi_{[t-1]}}\big)\le \epsilon_t+\kappa_t \, d(\xi_{[t-1]}, \tilde\xi_{[t-1]})\text{ for all } \tilde\xi_{[t-1]}. \nonumber
\end{align}
The ambiguity set~\eqref{eq:17} is constructed in the same way as~\eqref{eq:16} and contains all transitions which are close, in the sense specified, to the corresponding transitions of~$P$.

We consider the risk functional
\begin{equation}\label{nest-2b}
	\cR(Z)= \sup_{Q_1\in \M_1} \bbe_{Q_1}\left[ Z_1(\xi_1)\right],
\end{equation}
with conditional measures
\begin{equation}\label{int-recnest-2}
	Z_{t-1}(\xi_{[t-1]})= \sup_{Q_t\in\M_{t}^{\xi_{[t-1]}}}
	\left\{\bbe_{Q_t|\xi_{[t-1]}}[Z_t]= \int_{\Xi_t}
	Z_t(\xi_{[t-1]},\xi_t)\,Q_t(d\xi_t)\right\}
\end{equation}
on the nodes $\xi_{[t-1]}$. The risk measure~\eqref{nest-2b} is the risk measure defined in~\eqref{int-recnest-1}--\eqref{nest-2a} above with transitions specified in~\eqref{eq:17}.

Based on \cite[Proposition~4.26]{PflugPichlerBuch} we have the following result which allows comparing the risk measure with the simple expectation with respect to the reference probability measure.
\begin{proposition}\label{prop:53}
	It holds that \[|\rev{\mathfrak R}(Z)-\E_P [Z]|\rev{\le \nd_1(P,Q) \le } L_Z \sum_{t=1}^T\epsilon_t w_t\prod_{s=t+1}^T(1+w_s\kappa_s)\]
	\rev{for every $Q$ with conditional marginals in $\mathcal M_t^{\xi_{[t-1]}}$, $t=1,\dots,T$} and provided that the objective~$Z$ is Lipschitz with
	\[|Z(\xi_1,\dots,\xi_T)-Z(\tilde\xi_1,\dots,\tilde\xi_T)| \le \sum_{t=1}^T w_t\,d_t(\xi_t,\tilde \xi_t).\]
\end{proposition}
A notable situation arises for the stagewise independent measure $P=P_1\times\dots\times P_T$ in the rectangular case.

\begin{corollary}

	Suppose that
	\begin{equation}\label{eq:19}
		Q_t\in\M_t \text{ iff } d_1(P_t,Q_t)\le \epsilon_t.
	\end{equation}
	Then
 \[
 \left|\rev{\mathfrak R}(Z)-\E_P[Z]\right| \le L_Z\sum_{t=1}^T \epsilon_t\,w_t.
 \]
\end{corollary}

Similarly to~\eqref{eq:Ball}, the interpretation of~\eqref{eq:17} and~\eqref{eq:19} in the multistage setting is based on ambiguity sets of conditional transitions. The results in Proposition~\ref{prop:53} and the corollary compare the ambiguous approach with the  risk neutral approach and measure the impact of considering risk.

\section{Distributionally robust multistage optimization}
\label{sec:MSO}

Following the general paradigm of this paper we consider in this section  different approaches to formulation of    distributionally robust   multistage stochastic optimization problems. Again, we observe conceptual differences which we highlight here.

Consider the $L_p$ setting and  a composite functional $\cR$ of the form~\eqref{nestdec-2}. The involved $\R_{|\F_t}$ can be the  respective conditional distributionally robust functionals or the nested risk measures  associated with a law invariant coherent risk measure. In both cases there is a set $\widehat{\cM}$ of probability measures such that $\cR$ can be represented in the form~\eqref{eq-dualnest-2} provided that $\cR\colon\Z\to \bbr$ is finite valued.

Consider the  multistage stochastic program
\begin{equation}\label{stoc-1}
	\min\limits_{\pi\in \Pi}\cR (Z^\pi ),
\end{equation}
where $\Pi$ denotes the set of policies $\pi=\left (x_1,x_2(\xi_{[2]}),\dots,x_T(\xi_{[T]})\right)$ satisfying the feasibility constraints
\begin{equation}\label{stoc-2}
	x_1\in \X_1,\;x_t  \in \X_t(x_{t-1},\xi_t), \;t=2,\dots,T-1,
\end{equation}
and such that the total cost
\begin{equation}\label{stoc-3}
	\begin{array}{ll}
		Z^\pi:=  \sum_{t=1}^T f_t(x_t^\pi,\xi_t )
	\end{array}
\end{equation}
belongs to the space $\Z$ on which the functional $\cR$ is defined.
Here $x_t\in
\bbr^{n_t}$, $t=1,{\dots},T$,   $f_t\colon\bbr^{n_{t}}\times \bbr^{d_t}\to
\bbr$ are continuous functions and
$\X_t: \bbr^{n_{t-1}}\times \bbr^{d_t}  \rightrightarrows
\bbr^{n_{t}}$, $t=2,{\dots},T$,  are measurable multifunctions. The first stage
data, i.e.,  the vector $\xi_1$,  the function
$f_1\colon\bbr^{n_1}\to\bbr$, and the
set  $\X_1\subset \bbr^{n_1}$ are deterministic. The feasibility constraints in~\eqref{stoc-2} should be satisfied   almost surely   with respect to the reference measure~$P$.

\begin{rev}
Consider the nested \emph{risk} functional $\cR$ discussed in section~\ref{sec-nestrect}.
We can view the corresponding   risk averse  stochastic program~\eqref{stoc-1} as a stochastic game. In the framework of Markov decision processes (MDP), stochastic games were introduced in Shapley~\cite{Shapley} and were studied  extensively (see, e.g., the survey \cite{Jask2016}). For a policy $\pi\in \Pi$, we can think about the opponent who chooses at
every stage $t$  a distribution $Q^\pi_t\in \M_t^{\xi_{[t-1]}}$
for given  realization $\xi_{[t-1]}$ of the data process and  decision  $x_t=x_t^\pi$.
A choice of such probability distributions defines the corresponding policy of the opponent. For $\pi\in \Pi$,  denote by $\Gamma^\pi$ the set of such  policies of the opponent. Then the risk averse  program~\eqref{stoc-1} can be written in the following minimax form
\begin{equation}\label{minimax}
 \min_{\pi\in \Pi}\sup_{\gamma\in \Gamma^\pi}\bbe^{\pi,\gamma}[Z^\pi],
\end{equation}
where $Z^\pi$ is the total cost~\eqref{stoc-3}, $\gamma\in \Gamma^\pi$  is policy of the opponent with the respective distributions $Q^\pi_t\in \M_t^{\xi_{[t-1]}}$,  and
\begin{equation}\label{expv}
 \bbe^{\pi,\gamma}\left[\,\cdot\,
\right]:= \bbe_{Q^\pi_1}\left [ \bbe_{Q^\pi_2|\xi_{[1]}} \big [  \cdots \bbe_{Q^\pi_T|\xi_{[T-1]}}[\,\cdot \,]\big] \right],
 \end{equation}
 is the corresponding expectation  written in the nested (composite) form. Therefore  for $\pi\in \Pi$ we have that 
 \begin{equation}\label{game}
 \sup_{\gamma\in \Gamma^\pi}\bbe^{\pi,\gamma}[Z^\pi]=\cR(Z^\pi).
 \end{equation}
\end{rev}

Consider  the \emph{rectangular} case, i.e., assume that the set $\cM$ is given in the form~\eqref{eq-rect}. Then the nested and the conditional distributionally robust approaches are equivalent, and
the respective dynamic programming equations for value (cost-to-go) functions can be written as
\begin{align}\label{mul4-r}
	V_{t}\left (x_{t-1},\xi_{t}\right )&=\inf\limits_{x_t\in \X_t(x_{t-1},\xi_t)}
		\left\{ f_t(x_{t},\xi_t)+
		\V_{t+1}\left (x_t\right ) \right\},\\
	\label{mul5-r}
	\V_{t+1}\left (x_t\right )& =\sup\limits_{Q_{t+1}\in \M_{t+1}}
	\bbe_{Q_{t+1}}\left[V_{t+1}\left (x_t,\xi_{t+1}\right )\right],
\end{align}
$t=1, \dots,T$, with $\V_{T+1}(\cdot)\equiv 0$ and  $\X_1(x_{0},\xi_1)\equiv \X_1$.
A sufficient condition for
 $\bar{x}_t= \pi_t(\bar{x}_{t-1},\xi_t)$ to be an optimal policy is
\begin{equation}\label{mul5-cond}
	\bar{x}_t\in \argmin\limits_{x_t\in \X_t(\bar{x}_{t-1},\xi_t)}
	\left\{ f_t(x_{t},\xi_t)+
\V_{t+1}\left (x_t\right )\right\}.
\end{equation}
This condition is also necessary if the corresponding  functional is \emph{strictly} monotone. \begin{rev}Here policy of  the opponent is defined by a  choice of distribution $Q^\pi_t\in \M_t$,  at every stage of the process, which  depends  on decision $x_t=x_t^\pi$ but not on $\xi_{[t-1]}$. Again the corresponding distributionally robust/risk averse problem~\eqref{stoc-1} can be written in the minimax form~\eqref{minimax}.

In order to derive a dual of the risk averse multistage problem~\eqref{stoc-1} the `min' and `max' operators in the representation~\eqref{minimax} cannot be interchanged, even in the rectangular case,  since the set $\Gamma^\pi$ depends on $\pi\in \Pi$.
\end{rev}
On the  other hand by~\eqref{eq-dualnest-2} we can write problem~\eqref{stoc-1} in the form
\begin{equation}\label{mul-5}
\min_{\pi\in \Pi}\sup_{Q\in \widehat{\cM}} \bbe_{Q} [Z^\pi ].
\end{equation}
The dual of the minimax  problem~\eqref{mul-5} is obtained by interchanging the `min' and `max' operators, that is
\begin{equation}\label{mul-5dual}
 \max_{Q\in \widehat{\cM}} \inf_{\pi\in \Pi} \bbe_{Q} [Z^\pi ].
\end{equation}
\rev{The optimal value of  dual problem~\eqref{mul-5dual} is always less than or equal to the optimal value of the primal problem.
}
If the problem~\eqref{stoc-1} is convex, e.g., the objective functions $f_t(\cdot,\xi_t)$ are convex and the constraints are linear, then under certain  regularity conditions the optimal values of problems~\eqref{mul-5} and~\eqref{mul-5dual} are equal to each other.
\begin{rev}
\begin{remark}[Strong duality]
	Consider the $L_p$ setting with $\Z=L_p(\O,\F,P)$. Suppose that the problem is convex and  the cost functions $f_t(x_t,\xi_t)$  are continuous in $\xi_t$, and the  composite functional $\cR$ is finite valued and can be represented in the form~\eqref{eq-dualnest-2} (see Proposition~\ref {pr-dualrepnest}). The corresponding set $\widehat{\cA}$ is bounded, weakly$^*$  closed and hence is compact in the weak$^*$ topology of the dual space $\Z^*$.
	Consequently, it is possible  to apply Sion's minimax  theorem \cite{sion} to establish the required  no duality gap  property.
	To this end note that the  bilinear form $\bbe_\zeta[Z]=\lan Z,\zeta\ran$ is weakly$^*$ continuous in $\zeta\in \Z^*$, and  the required lower semi-continuity with respect to $Z\in \Z$  follows by Fatou's lemma.
\end{remark}
\end{rev}

As it was discussed before, even in the rectangular case the set $\cM$, defined in~\eqref{eq-rect}, in general is  different (smaller) than the corresponding set $\widehat{\cM}$.  By replacing the set $\widehat{\cM}$ in~\eqref{mul-5} with the set $\cM$ we obtain the problem
\begin{equation}\label{mul-5mm}
	 \min_{\pi\in \Pi}\left\{\R(Z^\pi)=\sup_{Q\in  \cM} \bbe_{Q} [Z^\pi ]\right\},
\end{equation}
and thus, comparing with~\eqref{mul-5}, that
\begin{equation}
	 \min_{\pi\in \Pi} \R(Z^\pi)
 \le  \min_{\pi\in \Pi} \cR(Z^\pi).
\end{equation}

\begin{rev}
	The Formulation~\eqref{mul-5mm} of multistage stochastic optimization problems was investigated by various authors. Note, however, that such a formulation is not adjusted to the dynamics of the decision process even in the rectangular setting.
\end{rev}
Suppose that the maximum in the right-hand side of~\eqref{mul5-r} is attained at a probability measure of the set $\M_{t+1}$. This probability measure depends on $x_t$. Since for a considered policy $\pi$, the decision  $x^\pi_t=x_t(\xi_{[t]})$ is a function of the history of the data process, the probability measure $Q_{t+1}^{\xi_{[t]}}$ maximizing the right-hand side of~\eqref{mul5-r}   also depends  on the history of the data process (compare with the  discussion of Exam\-ple~\ref{ex-twost}).
On the other hand for the problem~\eqref{mul-5mm}, of minimizing $\R(Z^\pi)$ over $\pi\in \Pi$, the corresponding marginal  probability measures $Q_t$ should not depend on realizations of the data process. As a consequence,  problem~\eqref{mul-5mm} cannot be decomposed into  dynamic programming equations of the form~\eqref{mul4-r}--\eqref{mul5-r}.

\section{Summary}
\label{sec:Summary}
The literature suggests various approaches towards constructing multistage  counterparts of distributionally robust  functionals.
\begin{rev}One approach is to use the “static” functional $\R$ which is not adjusted to the dynamics of the decision process.   The other approach, now routinely used in the risk averse stochastic optimization, is to employ  nested formulations of the corresponding risk measures. In such an approach the risk is controlled at every stage of the decision process while the total value of the corresponding objective function is irrelevant.
\end{rev}
Although natural in some sense, this paper exposes that the respective nested \emph{distributionally robust} approach could be  very different. It is demonstrated that in the case of the rectangular setting, which is important from an application perspective, both approaches are equivalent.

We investigate the foundations by addressing the individual features, the major differences and particularly the relevance of conditioned risk functionals for multistage stochastic optimization.

\bibliographystyle{siamplain}
\bibliography{reference}
\end{document}

%% file: ex_shared.tex
\usepackage{lipsum}
\usepackage{amsfonts}
\usepackage{graphicx}
\usepackage{epstopdf}
\usepackage{algorithmic}
\ifpdf
  \DeclareGraphicsExtensions{.eps,.pdf,.png,.jpg}
\else
  \DeclareGraphicsExtensions{.eps}
\fi

\newsiamremark{remark}{Remark}
\newsiamremark{hypothesis}{Hypothesis}
\crefname{hypothesis}{Hypothesis}{Hypotheses}
\newsiamthm{claim}{Claim}

\headers{Distributionally Robust Optimization}{A. Pichler, and  A. Shapiro}

\title{Mathematical Foundations of Distributionally Robust  Multistage Optimization\thanks{Submitted to the editors DATE.
}}

\author{Alois Pichler\thanks{Technische Universit\"{a}t Chemnitz, Fakult\"{a}t für Mathematik, D--09111 Chemnitz, Germany. Funded by Deutsche Forschungsgemeinschaft (DFG, German Research Foundation)~-- Project-ID 416228727 -- SFB~1410, \email{alois.pichler@math.tu-chemnitz.de}.}
\and Alexander Shapiro\thanks{School of Industrial and Systems Engineering, Georgia Institute of Technology, Atlanta, Georgia 30332-0205, \email{ashapiro@isye.gatech.edu}.}
}

\usepackage{amsopn}
